\documentclass[12pt]{amsart}
\usepackage{graphicx}
\usepackage{amsmath,amssymb}
\usepackage{amsthm}

\newtheorem{thm}{Theorem}[section]
\newtheorem{cor}[thm]{Corollary}
\newtheorem{lem}[thm]{Lemma}

\theoremstyle{definition}

\theoremstyle{remark}
\newtheorem{rem}[thm]{Remark}
\theoremstyle{example}

\theoremstyle{conjecture}

\numberwithin{equation}{section}

\newcommand{\eps}{\varepsilon}

\newcommand{\del}{\delta}

\def\dlim{\displaystyle\lim}
\def\begr{\begin{eqnarray}} \def\endr{\end{eqnarray}}
\def\msk{\medskip}

\newcommand{\BB}{{\mathbb B}}
\newcommand{\DD}{{\mathbb D}}
\newcommand{\RR}{{\mathbb R}}
\newcommand{\SSS}{{\mathbb S}}
\newcommand{\CC}{{\mathbb C}}

\newcommand{\NN}{{\mathbb N}}

\newcommand{\re}{{\rm Re}\,}

\begin{document}

\title[Hadamard gap series]{Hadamard gap series in weighted-type spaces on the unit ball}%

\date{\today}%

\author{  Bingyang Hu and Songxiao Li  }%
\address{Bingyang Hu: Department of Mathematics, University of Wisconsin, Madison, WI 53706-1388, USA.}%
\email{BingyangHu@math.wisc.edu}

\address{Songxiao Li:  Institute of Systems Engineering,
Macau University of Science and Technology,  Avenida Wai Long,
Taipa, Macau. }%
\email{jyulsx@163.com }




\maketitle

\begin{abstract}
We give a sufficient and necessary condition for an analytic function $f(z)$ on the unit ball $\BB$ in $\CC^n$ with Hadamard gaps, that is, for $f(z)=\sum_{k=1}^\infty P_{n_k}(z)$ where $P_{n_k}(z)$ is a homogeneous polynomial of degree $n_k$ and $n_{k+1}/n_k \ge c>1$ for all $k \in \NN$, to belong to the weighted-type space $H^\infty_\mu$ and the corresponding little weighted-type space $H^\infty_{\mu, 0}$, under some condition posed on the weighted funtion $\mu$. We also   study the growth rate of those functions in $H^\infty_\mu$. Finally, we characterize the boundedness and compactness of weighted composition operator from weighted-type space  $H^\infty_\mu$ to mixed norm spaces.
\end{abstract}

\bigskip
\section{Introduction}

Let $\BB$ be the open unit ball in $\CC^n$ with $\SSS$ as its boundary
and $H(\BB)$ the collection of all holomorphic functions in $\BB$.
$H^{\infty}(\BB)$ denotes the Banach space consisting of all bounded
holomorphic functions in $\BB$ with the norm $\|f\|_{\infty}=\sup\limits_{z \in \BB} |f(z)|$.

A positive continuous function $\mu$ on $[0, 1)$ is called \emph{normal} if there exists
positive numbers $\alpha$ and $\beta$, $0<\alpha<\beta$, and $\del \in (0, 1)$
such that (see, e.g., \cite{SW})
\begr \begin{split}
&\frac{\mu(r)}{(1-r)^{\alpha}} \  \textrm{is decreasing on}  \ [\del, 1),
\quad \lim_{r \to 1} \frac{\mu(r)}{(1-r)^{\alpha}}=0,\\
& \frac{\mu(r)}{(1-r)^{\beta}} \  \textrm{is increasing on}  \ [\del, 1),
\quad \lim_{r \to 1} \frac{\mu(r)}{(1-r)^{\beta}}=\infty.
\end{split} \endr
Note that a normal function $\mu: [0, 1) \to [0, \infty)$ is decreasing in a neighborhood of $1$ and satisfies $\lim\limits_{r \to 1^{-}}\mu(r)=0$.

An $f \in H(\BB)$ is said to belong to the weighted-type space, denoted by $H_\mu^\infty=H_\mu^\infty(\BB)$ if
$$
\|f\|=\sup_{z \in \BB} \mu(|z|)|f(z)|<\infty,
$$
where $\mu$ is normal on $[0, 1)$ (see, e.g. \cite{SS1}). It is well-known that $H_\mu^\infty$ is a Banach space with the norm $\| \cdot \|$.

The little weighted-type space, denoted by $H_{\mu, 0}^\infty$, is the closed subspace of $H_\mu^\infty$ consisting of those $f \in H_\mu^\infty$ such that
$$
\lim_{|z| \to 1^{-}} \mu(|z|)|f(z)|=0.
$$
When $\mu(|z|)=(1-|z|^2)^{\alpha}, \alpha>0$, the induced spaces $H^\infty_\mu$ and $H^\infty_{\mu, 0}$ become the Bers-type space and little Bers-type space respectively.

Let $\phi$ be a normal function on $[0,1)$. For $0<p,q<\infty$, the
\emph{mixed-norm space $H(p,q,\phi)=H(p,q,\phi)(\mathbb{B})$} is the space  consisting
of all $f\in H(\mathbb{B})$ such that
 \begr
  \|f\|_{H(p,q,\phi)}=\left(\int_0^1M^p_q(f,r)\frac{\phi^{p}(r)}{1-r}
dr\right)^{1/p}<\infty,\nonumber
\endr
where
$$M_q(f,r)=\bigg(\int_\SSS |f(r\zeta)|^qd\sigma(\zeta)\bigg)^{1/q}, $$
and   $\sigma$ is the normalized area measure on $\SSS$.

Let $\varphi$ be a holomorphic self-map of $\mathbb{B}$ and $u\in H(\mathbb{B})$. For $f\in H(\mathbb{B})$, the \emph{weighted composition operator} $uC_\varphi$ is defined by
  \begr
   (u C_\varphi f)(z)=u(z)f(\varphi(z)),\quad z\in \mathbb{B}.\nonumber
  \endr
The weighted composition operator can be regarded as a generalization of the \emph{multiplication operator} and the \emph{composition operator}, which are defined by $M_u (f)=(uf)(z)$ and $(C_\varphi f)(z)= f(\varphi(z))$, respectively. See \cite{cm} for more information on this topic.

We say that an $f\in H(\BB)$ has the \emph{Hadamard gaps} if
$$
f(z)=\sum_{k=0}^{\infty} P_{n_k}(z),
$$
where $P_{n_k}$ is a homogeneous polynomial of degree $n_k$ and there exists some $c>1$ (see. e.g.,
\cite{SS}),
$$
 \frac{n_{k+1}}{n_k}\ge c,\ \forall k \ge 0.
$$

Hadamard gap series on spaces of holomorphic functions in the unit disc $\DD$ or
in the unit ball $\BB$ has been studied quite well. We refer the readers to the
related results in \cite{JSC, LS, JM, SS, SS1,
WZ, SY, YX, zz, Zhu2} and the reference therein.

In \cite{YX}, the authors studied the Hadamard gap series and the growth rate of the functions  in $H^\infty_\mu$ in the unit disk. Motivated by \cite{YX},
the aim of this paper is to study the Hadamard gap series in $H^\infty_\mu$, as well as its little space $H^\infty_{\mu, 0}$ on the unit ball. Moreover, as an application of our main result, we characterize the growth rate of those functions in $H^\infty_\mu$.  Finally, we give some sufficient and necessary conditions for the boundedness and compactness of weighted composition operators from weighted-type space  $H^\infty_\mu$ to mixed norm spaces.

Througout this paper, for $a, b \in \RR$, $a \lesssim b$ ($a \gtrsim b$, respectively) means there exists a positive number $C$, which is independent of $a$ and $b$, such that $a \leq Cb$ ($ a \geq Cb$, respectively). Moreover, if both $a \lesssim b$ and $a \gtrsim b$ hold, then we say $a \simeq b$.

\section{Hadamard gap series in $H^\infty_\mu$ and $H^\infty_{\mu, 0}$}
Let $f(z)=\sum\limits_{k=0}^\infty P_k(z)$ be  a holomorphic function in $\BB$, where $P_k(z)$ is a homogeneous polynomial with degree $k$. For $k\ge 0$, we denote
$$
M_k=\sup_{\xi \in \SSS} |P_k(\xi)|.
$$

We have the following estimations on $M_k$ of a holomorphic function $f \in H^\infty_{\mu}$ (or $f \in H^\infty_{\mu, 0}$, respectively).

\begin{thm} \label{thm001}
Let $\mu$ be a normal function on $[0,1)$.  Let $f(z)=\sum\limits_{k=0}^\infty P_{k}(z), z \in \BB$. Then the following statements hold.
\begin{enumerate}
\item If $f \in H^\infty_\mu$, then $\sup\limits_{k \ge 0} M_k \mu\left(1-\frac{1}{k}\right)<\infty$. \\
\item If $f \in H^\infty_{\mu, 0}$, then $\lim\limits_{k \to \infty} M_k \mu\left(1-\frac{1}{k}\right)=0$.\\
\end{enumerate}
\end{thm}

\begin{proof}
(1). Suppose $f \in H^\infty_\mu$.  Fix a $\xi \in \SSS$ and denote
$$
f_\xi(w)=\sum_{k=0}^\infty P_{k}(\xi) w^k=\sum_{k=0}^\infty
P_{k}(\xi w), ~~ w \in \DD.
$$
Since $f \in H(\BB)$, it is known that for a fixed $\xi \in
\SSS$, $f_\xi(w)$ is holomorphic in $\DD$ (see, e.g., \cite{Rud}).
Hence, for any $r \in (0,
1)$, we have \newpage
\begin{eqnarray} \label{ineq0500}
M_k%
&=& \sup_{ \xi \in   \SSS} |P_k(\xi)|=\sup_{ \xi \in   \SSS} \left |\frac{1}{2\pi i} \int_{|w|=r} \frac{f_\xi(w)}{w^{k+1}}dw\right| \\
&=& \frac{1}{2\pi} \sup_{ \xi \in \SSS}\left |\int_{|w|=r} \frac{f( \xi w)}{w^{k+1}}dw\right|  \nonumber \\
&\le& \frac{1}{2\pi}\sup_{ \xi \in   \SSS} \int_{|w|=r} \frac{|f( \xi w)|}{r^{k+1}}|dw|  \nonumber \\
&=& \frac{1}{2\pi}\sup_{ \xi \in  \SSS} \int_{|w|=r} \frac{|f( \xi w)| \mu(|\xi w|)}{r^{k+1} \mu(r)}|dw| \nonumber \\
&\le& \frac{\|f\|}{r^k\mu(r)}. \nonumber
\end{eqnarray}
In \eqref{ineq0500}, letting $r=1-\frac{1}{k}, k \ge 2, k \in \NN$, we have
$$
M_k \le \frac{\|f\|}{(1-\frac{1}{k})^k\mu\left(1-\frac{1}{k}\right)}.
$$
Thus, for each $k \ge 2$,
$$
M_k\mu\left(1-\frac{1}{k}\right) \le \frac{\|f\|}{(1-\frac{1}{k})^k} \le 4\|f\|,
$$
which implies that
$$
\sup_{k \ge 1} M_k\mu\left(1-\frac{1}{k}\right) \le \max \left\{\mu(0)M_1,  4\|f\| \right\}<\infty.
$$

(2). Suppose $f \in H^\infty_{\mu, 0}$, that is, for any $\varepsilon>0$, there exists a $\del \in (0, 1)$, when $\del<|z|<1$,
$$
\mu(|z|)|f(z)|< \varepsilon.
$$
Take $N_0 \in \NN$ satisfying $\del<1-\frac{1}{k}<1$ when $k>N_0$.
Then for any $k > N_0$ and $r=1-\frac{1}{k}$, as the proof
in the previous part, we have
$$
M_k \le \frac{1}{ (1-\frac{1}{k})^k \mu\left(1-\frac{1}{k}\right)}  \cdot \sup_{
\del<|z|<1}  \mu(|z|)|f(z)|   <\frac{\varepsilon}{ (1-\frac{1}{k})^k \mu\left(1-\frac{1}{k}\right)},
$$
which implies
$$
 M_k \mu\left(1-\frac{1}{k}\right) \le \frac{\varepsilon}{ (1-\frac{1}{k})^k}
\le 4 \varepsilon, ~~k>N_0.
$$
 Hence we have  $\dlim_{k \to \infty} M_k\mu\left(1-\frac{1}{k}\right)=0$.
\end{proof}

\begin{thm} \label{thm002}
Let $\mu$ be a normal function on $[0,1)$.  Let $f(z)=\sum\limits_{k=0}^\infty P_{n_k}(z)$ with Hadamard gaps, where $P_{n_k}$ is a homogeneous polynomial of degree $n_k$. Then the following assertions hold.
\begin{enumerate}
\item $f \in H^\infty_\mu$ if and only if $\sup\limits_{k \ge 1} \mu\left(1-\frac{1}{n_k}\right) M_{n_k}<\infty$.
\item $f \in H^\infty_{\mu, 0}$ if and only if $\lim\limits_{k \to \infty} \mu\left(1-\frac{1}{n_k}\right) M_{n_k}=0$.
\end{enumerate}
\end{thm}

\begin{proof}
 By Theorem \ref{thm001}, it suffices to show the sufficiency of both statements.

(1). Noting that
$$
|f(z)|=\left|\sum_{k=0}^\infty P_{n_k} \left(\frac{z}{|z|}\right) |z|^{n_k}\right| \leq \sum_{k=0}^{\infty} M_{n_k}|z|^{n_k} \lesssim \sum_{k=0}^{\infty} \frac{|z|^{n_k}}{\mu\left(1-\frac{1}{n_k}\right)},
$$
from the proof of  \cite[Theorem 2.3]{YX}, we have
\begin{eqnarray*}
 \frac{|f(z)| }{1-|z|}&\lesssim&  \sum_{m=1}^\infty\Big( \sum_{n_k\leq m} \frac{1}{\mu\left(1-\frac{1}{n_k}\right)} \Big)|z|^m\lesssim \sum_{m=1}^\infty  \frac{|z|^m}{\mu\left(1-\frac{1}{m}\right)}\\
&\lesssim& \frac{1}{(1-|z|)\mu(|z|)} ,
\end{eqnarray*}
which implies $f \in H^\infty_\mu$, as  desired.

(2).  Since $\lim\limits_{k \to \infty} \mu\left(1-\frac{1}{n_k}\right) M_{n_k}=0$, we have $\sup\limits_{k \ge 1} \mu\left(1-\frac{1}{n_k}\right) M_{n_k}<\infty$. Hence by part (1), we have $f \in H^\infty_\mu$. For any $\varepsilon>0$, there exists a $N_0 \in \NN$
satisfying when $m>N_0$,
$$
M_{n_m} \mu\left(1-\frac{1}{n_m}\right)<\varepsilon.
$$
For each $m \in \NN$, put $f_m(z)=\sum\limits_{k=0}^m
P_{n_k}(z)$. Note that
\begin{eqnarray*}
\mu(|z|)|f_m(z)|%
&\le& \mu(|z|)\bigg(\sum_{k=0}^m |P_{n_k}(z)|\bigg) \\
&=&\mu(|z|)\left(\sum_{k=0}^m \left| P_{n_k}\left(\frac{z}{|z|}\right) |z|^{n_k}\right|\right)\\
&\le& K_m\mu(|z|) \sum_{k=0}^m |z|^{n_k} \le mK_m\mu(|z|),
\end{eqnarray*}
where $K_m=\max\{M_{n_0}, M_{n_1}, M_{n_2}, \dots, M_{n_m}\}$. Noting that $\lim\limits_{|z| \to 1^{-}} \mu(|z|)=0$, we have $\lim\limits_{|z| \to 1^{-}} \mu(|z|)|f_m(z)|=0$, which implies for each $m \in \NN$, $f_m \in H^\infty_{\mu, 0}$. Hence it suffices to show that $\|f_m-f\| \to 0$ as $m \to \infty$. Indeed, for $m>N_0$, we have
$$
|f_m(z)-f(z)|= \left| \sum_{k=m+1}^\infty P_{n_k}(z)\right| \le \sum_{k=m+1}^\infty M_{n_k} |z|^{n_k}\le \varepsilon \sum_{k=m+1}^\infty  \frac{|z|^{n_k}}{\mu\left(1-\frac{1}{n_k}\right)}.
$$
From this, the result easily follows from the proof of part (1).
\end{proof}

\bigskip
\section{Growth rate}

As an application of Theorem \ref{thm002}, in this section, we show the following result.

\begin{thm}\label{thm003}
Let $\mu$ be a normal function on $[0,1)$.  Then there exists a positive integer $M=M(n)$ with the following property: there exists $f_i \in H^\infty_\mu, 1 \le i \le M$, such that
$$
\sum_{i=1}^M |f_i(z)| \gtrsim \frac{1}{\mu(|z|)}, \quad z \in \BB.
$$
\end{thm}

Note that the result in \cite[Theorem 2.5]{YX} in the unit disc is a particular case of Theorem \ref{thm003} when $n=1$.

\begin{rem}
We observe that $M$ cannot be $1$. Indeed, assume that there exists a $f \in H^\infty_\mu$, such that
$$
|f(z)| \gtrsim \frac{1}{\mu(|z|)}, \quad z \in \BB.
$$
It implies that $f(z)$ has no zero in $\BB$, and it follows that there exists $g \in H(\BB)$, such that $f=e^g$. Thus,
$$
|f(z)|=\left|e^{g(z)}\right|=e^{\re g(z)},
$$
which implies that $e^{\re g(z)} \gtrsim \frac{1}{\mu(|z|)}$ and hence $\re g(z) \gtrsim \log \frac{1}{\mu(|z|)}$. For each $r \in (0, 1)$, integrating on both sides of the above inequality on $r\SSS=\{z \in \BB, |z|=r\}$, we have
$$
\int_{r\SSS} \re g(z) d\sigma \gtrsim \int_{r \SSS} \log \left( \frac{1}{\mu(|z|)} \right) d\sigma=\log \left(\frac{1}{\mu(r)}\right) \cdot \sigma(r \SSS).
$$
 By the mean value property, we have $\re g(0) \gtrsim \log \left(\frac{1}{\mu(r)}\right),  \forall r \in (0, 1)$, which is impossible.
\end{rem}

Before we formulate the proof of our main result, we need some preliminary results. In the sequel, for $\xi, \zeta \in \SSS$, denote
$$
d(\xi, \zeta)=(1-|\langle \xi, \zeta \rangle|^2)^{1/2}.
$$
Then $d$ satisfies the triangle inequality (see, e.g., \cite{ABA}). Moreover, we write $E_\del(\zeta)$  for the $d$-ball with radius $\del \in (0, 1)$ and center at $\zeta \in \SSS$:
$$
E_\del(\zeta)=\left\{ \xi \in \SSS: d(\xi, \zeta)<\del\right\}.
$$
We say that a subset $\Gamma$ of $\SSS$ is \emph{d-sperated} by $\del>0$, if $d$-balls with radius $\del$ and center at points of $\Gamma$ are pairwise disjoint.

We begin with several lemmas, which play important role in the proof of our main result.

\begin{lem} \label{lem01}\cite{CR, DU} For each $a>0$, there exists a positive integer $M=M_n(a)$ with the following property: if $\del>0$, and if $\Gamma \subset \SSS$ is $d$-seperated by $a\del$, then $\Gamma$ can be decomposed into $\Gamma=\bigcup\limits_{j=1}^M \Gamma_j$ in such a way that each $\Gamma_j$ is $d$-seperated by $\del$.
\end{lem}

\begin{lem} \label{lem02}\cite[Lemma 2.3]{CR}
Suppose that $\Gamma \subset \SSS$ is $d$-seperated by $\del$ and let $k$ be a positive integer. If
$$
P(z)=\sum_{\zeta \in \Gamma} \langle z, \zeta \rangle^k, \quad  z \in \BB,
$$
then
$$
|P(z)| \le 1+\sum_{m=1}^\infty (m+2)^{2n-2}e^{-m^2\del^2k/2}.
$$
\end{lem}

\textbf{Proof of Theorem \ref{thm003}.}
 We will prove the theorem by constructing $f_i \in H^\infty_\mu$ satisfying the given property only near the boundary (then, by adding a proper constant, one obtains the given property on all of the unit ball). Since $\mu$ is normal, by the definition of normal function, there exists positive numbers $\alpha, \beta$ with $0<\alpha<\beta$, and $\delta \in (0, 1)$ satisfy (1.1).  Take and fix some small positive number $A<1$ such that
\begin{equation} \label{eqthm030-1}
\sum_{m=0}^\infty (m+2)^{2n-2} e^{\frac{-m^2}{2A^2}} \le \frac{1}{27}.
\end{equation}
Let $M=M_n\left(\frac{A}{2}\right)$ be a positive integer provided by Lemma \ref{lem01} with $A/2$ in place of $a$. Let $p$ be a sufficiently large positive integer so that
\begin{equation} \label{eqthm0300}
1-\frac{1}{p} \ge \delta,
\end{equation}
\begin{equation} \label{eqthm0301}
\frac{1}{3} \le \left(1- \frac{1}{p}\right)^p \le \frac{1}{2},
\end{equation}
\begin{equation} \label{eqthm0302}
\frac{1}{p^{\alpha M}-1} \le \frac{1}{200},
\end{equation}
and
\begin{equation} \label{eqthm0303}
\frac{p^{\beta M} \cdot 2^{-p^{M-0.5}}}{1-p^{\beta M} \cdot 2^{-(p^{2M-0.5}-p^{M-0.5}})} \le \frac{1}{200}.
\end{equation}
For each postive integer $j \le M$, set $\del_{j, 0}$ such that
\begin{equation} \label{eqthm0304}
A^2p^j \del_{j, 0}^2=1
\end{equation}
and inductively choose $\del_{j ,v}$ such that
\begin{equation} \label{eqthm0305}
p^M\del_{j, v}^2=\del_{j, v-1}^2, \quad v=1,2, \dots.
\end{equation}
From \eqref{eqthm0304} and \eqref{eqthm0305}, we get
\begin{equation} \label{eqthm0306}
A^2p^{vM+j}\del_{j, v}^2=1.
\end{equation}
For each fixed $j$ and $v$, let $\Gamma^{j, v}$ be a maximal subset of $\SSS$ subject to the condition that $\Gamma^{j, v}$ is $d$-seperated by $A\del_{j, v}/2$. Then by Lemma \ref{lem01},   write
\begin{equation} \label{eqthm0307}
\Gamma^{j, v}=\bigcup_{l=1}^M \Gamma_{j, vM+l}
\end{equation}
in such a way that each $\Gamma_{j, vM+l}$ is $d$-seperated by $\del_{j, v}$.

For each $i, j=1,2, \dots, M$ and $v \ge 0$, set
$$
P_{i, vM+j}(z)=\sum_{\xi \in \Gamma_{j, vM+\tau^i(j)}} \langle z, \xi \rangle^{p^{vM+j}},
$$
where $\tau^i$ is the $i^{th}$ iteration of the permutation $\tau$ on $\{1, 2, \dots, M\}$ defined by
\[ \tau(j)=\begin{cases}
j+1, & j<M;\\
1, & j=M.
\end{cases} \]

 By \eqref{eqthm0306},  Lemma \ref{lem02} and \eqref{eqthm030-1}, we get  that
\begin{eqnarray} \label{eqthm0309}
|P_{i, vM+j}(z)|%
&\le& 1+\sum_{m=1}^\infty (m+2)^{2n-2}e^{-m^2\del_{j, v}^2p^{vM+j}/2} \\
&\le& 1+ \sum_{m=1}^\infty (m+2)^{2n-2} e^{-\frac{m^2}{2A^2}} \le 2, \quad z \in \BB, \nonumber
\end{eqnarray}
for all $i, j=1,2, \dots, M$ and $v \ge 0$.

Define
$$
g_{i, j}(z)=\sum_{v=0}^\infty \frac{P_{i, vM+j}(z)}{\mu\left(1-\frac{1}{p^{vM+j}}\right)}, \quad z \in \BB.
$$
By Theorem \ref{thm002},   it is clear that for each $i,j \in \{1,2, \dots, M\}$, $g_{i, j} \in H^\infty_\mu$.

We will show that for every $v \ge 0, 1 \le j \le M$ and $z \in \BB$ with that
\begin{equation} \label{eqthm0308}
1-\frac{1}{p^{vM+j}} \le |z| \le 1-\frac{1}{p^{vM+j+\frac{1}{2}}},
\end{equation}
there exists an $i \in \{1, 2, \dots, M\}$ such that
$$
|g_{i, j}(z)| \ge \frac{C}{\mu(|z|)},
$$
where $C$ is some constant independent of the choice of $i, j$ and $z$.

Fix $v, j$ and $z$ for which \eqref{eqthm0308} holds. Let $z=|z| \eta$ where $\eta \in \SSS$. Since $d$-balls with radius $A\del_{j, v}$ and centers at points of $\Gamma^{j, v}$ cover $\SSS$ by maximality, there exists some $\zeta \in \Gamma^{j, v}$ such that
$$
\eta \in E_{A\del_{j, v}}(\zeta).
$$
Note that $\zeta \in \Gamma_{j, vM+l}$ for some $1 \le l \le M$ by \eqref{eqthm0307} and hence $\zeta \in \Gamma_{j, vM+\tau^i(j)}$ for some $1 \le i \le M$.

We now estimate $|g_{i, j}(z)|$. By \eqref{eqthm0309},
\begin{eqnarray*}
|g_{i, j}(z)|%
&=& \left| \sum_{k=0}^\infty \frac{P_{i, kM+j}(z)}{\mu\left(1-\frac{1}{p^{kM+j}}\right)} \right| \\
&\ge& \left| \frac{P_{i, vM+j}(z)}{\mu\left(1-\frac{1}{p^{vM+j}}\right)}\right|-\left| \sum_{k \neq v} \frac{P_{i, kM+j}(z)}{\mu\left(1-\frac{1}{p^{kM+j}}\right)} \right| \\
&=& \frac{|z|^{p^{vM+j}}|P_{i, vM+j}(\eta)|}{\mu\left(1-\frac{1}{p^{vM+j}}\right)}-\left| \sum_{k \neq v} \frac{|z|^{kM+j}P_{i, kM+j}(\eta)}{\mu\left(1-\frac{1}{p^{kM+j}}\right)} \right| \\
&\ge& \frac{|z|^{p^{vM+j}}|P_{i, vM+j}(\eta)|}{\mu\left(1-\frac{1}{p^{vM+j}}\right)}-2 \sum_{k=0}^{v-1} \frac{|z|^{p^{kM+j}}}{\mu\left(1-\frac{1}{p^{kM+j}}\right)}\\
&& -2 \sum_{k=v+1}^\infty \frac{|z|^{p^{kM+j}}}{\mu\left(1-\frac{1}{p^{kM+j}}\right)}\\
&=& I_1-I_2-I_3,
\end{eqnarray*}
where
$$
I_1=\frac{|z|^{p^{vM+j}}|P_{i, vM+j}(\eta)|}{\mu\left(1-\frac{1}{p^{vM+j}}\right)}, \quad I_2=2 \sum_{k=0}^{v-1} \frac{|z|^{p^{kM+j}}}{\mu\left(1-\frac{1}{p^{kM+j}}\right)}
$$
and
$$
I_3=2 \sum_{k=v+1}^\infty \frac{|z|^{p^{kM+j}}}{\mu\left(1-\frac{1}{p^{kM+j}}\right)}.
$$
Now we estimate $I_1, I_2$ and $I_3$ respectively.

$\bullet$ \textbf{Estimation of $I_1$.}

By \eqref{eqthm0301} and \eqref{eqthm0308}, we obtain
$$
|z|^{p^{vM+j}} \ge \left(1-\frac{1}{p^{vM+j}}\right)^{p^{vM+j}} \ge \frac{1}{3},
$$
and therefore
\begin{eqnarray*}
I_1%
&\ge& \frac{|P_{i, vM+j}(\eta)|}{3\mu\left(1-\frac{1}{p^{vM+j}}\right)}\\
&\ge& \frac{\left(|\langle \eta, \zeta \rangle|^{p^{vM+j}}-\sum_{\xi \in \Gamma_{j, vM+\tau^i(j)}, \ \xi \neq \zeta} |\langle \eta, \xi \rangle|^{p^{vM+j}}\right)}{3\mu\left(1-\frac{1}{p^{vM+J}}\right)}\\
&& \textrm{(By the proof of \cite[Theorem 2.1]{CR})}\\
&\ge& \frac{2}{27\mu\left(1-\frac{1}{p^{vM+J}}\right)}.
\end{eqnarray*}

$\bullet$ \textbf{Estimation of $I_2$.}

By the definition of normal function and \eqref{eqthm0300}, we have for each $s \in \NN$,
$$
\frac{\left(1-\left(1-\frac{1}{p^{sM+j}}\right)\right)^\alpha}{\left(1-\left(1-\frac{1}{p^{(s+1)M+j}}\right)\right)^\alpha} \le \frac{\mu\left(1-\frac{1}{p^{sM+j}}\right)}{\mu\left(1-\frac{1}{p^{(s+1)M+j}}\right)}\le \frac{\left(1-\left(1-\frac{1}{p^{sM+j}}\right)\right)^\beta}{\left(1-\left(1-\frac{1}{p^{(s+1)M+j}}\right)\right)^\beta},
$$
that is,
\begin{equation} \label{eqthm0310}
1<p^{M\alpha} \le\frac{\mu\left(1-\frac{1}{p^{sM+j}}\right)}{\mu\left(1-\frac{1}{p^{(s+1)M+j}}\right)} \le p^{M\beta}.
\end{equation}
Combining this with \eqref{eqthm0302}, we have
\begin{eqnarray*}
I_2%
&\le& 2 \sum_{k=0}^{v-1} \frac{1}{\mu\left(1-\frac{1}{p^{kM+j}}\right)}\\
&=& \frac{2}{\mu\left(1-\frac{1}{p^{vM+j}}\right)} \sum_{k=0}^{v-1} \bigg[ \frac{\mu\left(1-\frac{1}{p^{vM+j}}\right)}{\mu\left(1-\frac{1}{p^{(v-1)M+j}}\right)} \frac{\mu\left(1-\frac{1}{p^{(v-1)M+j}}\right)}{\mu\left(1-\frac{1}{p^{(v-2)M+j}}\right)} \dots  \\
&&\times \frac{\mu\left(1-\frac{1}{p^{(k+1)M+j}}\right)}{\mu\left(1-\frac{1}{p^{kM+j}}\right)}\bigg]\\
&\le& \frac{2}{\mu\left(1-\frac{1}{p^{vM+j}}\right)} \sum_{k=0}^{v-1} \frac{1}{p^{\alpha M(v-k)}}\\
&\le& \frac{2}{\mu\left(1-\frac{1}{p^{vM+j}}\right)} \cdot \frac{1}{p^{\alpha M}-1}\\
&\le& \frac{1}{100\mu\left(1-\frac{1}{p^{vM+j}}\right)}.
\end{eqnarray*}

$\bullet$ \textbf{Estimation of $I_3$.}

Noting that by \eqref{eqthm0301} and \eqref{eqthm0308}, we have
\begin{equation} \label{eqthm0311}
|z|^{p^{vM+j}} \le \left(1-\frac{1}{p^{vM+j+\frac{1}{2}}}\right)^{p^{vM+j+\frac{1}{2}} \cdot p^{-\frac{1}{2}}} \le \left(\frac{1}{2}\right)^{p^{-\frac{1}{2}}}.
\end{equation}

Hence, by \eqref{eqthm0303}, \eqref{eqthm0310} and \eqref{eqthm0311}, we have
\begin{eqnarray*}
I_3%
&=& \frac{2|z|^{p^{(v+1)M+j}}}{\mu\left(1-\frac{1}{p^{vM+j}}\right)} \cdot \sum_{k=v+1}^\infty \bigg[\frac{\mu\left(1-\frac{1}{p^{vM+j}}\right)}{{\mu\left(1-\frac{1}{p^{kM+j}}\right)}} |z|^{\left(p^{kM+j}-p^{(v+1)M+j}\right)}\bigg] \\
&=& \frac{2|z|^{p^{(v+1)M+j}}}{\mu\left(1-\frac{1}{p^{vM+j}}\right)} \cdot \sum_{k=v+1}^\infty \bigg[ \frac{\mu\left(1-\frac{1}{p^{vM+j}}\right)}{\mu\left(1-\frac{1}{p^{(v+1)M+j}}\right)}  \dots \frac{\mu\left(1-\frac{1}{p^{(k-1)M+j}}\right)}{\mu\left(1-\frac{1}{p^{kM+j}}\right)}\\
&&  |z|^{\left(p^{kM+j}-p^{(v+1)M+j}\right)}\bigg] \\
&\le&\frac{2|z|^{p^{(v+1)M+j}}}{\mu\left(1-\frac{1}{p^{vM+j}}\right)} \cdot \sum_{k=v+1}^\infty \left[p^{(\beta M)(k-v)} |z|^{\left(p^{kM+j}-p^{(v+1)M+j}\right)}\right]\end{eqnarray*}
\begin{eqnarray*}
&=& \frac{2|z|^{p^{(v+1)M+j}}}{\mu\left(1-\frac{1}{p^{vM+j}}\right)} \cdot \sum_{k=v+1}^\infty \left[p^{\beta M} p^{(\beta M)(k-v-1)} |z|^{p^j\left(p^{kM}-p^{(v+1)M}\right)}\right]\\
&& \textrm{(Let $s=k-v-1$.)}\\
&=&\frac{2|z|^{p^{(v+1)M+j}}}{\mu\left(1-\frac{1}{p^{vM+j}}\right)} \cdot \sum_{s=0}^\infty\left[p^{\beta M} p^{\beta Ms} |z|^{p^{j+(v+1)M}\left(p^{sM}-1\right)} \right]\\
&& \textrm{(By $p^{sM-1} \ge s(p^M-1)$, where $s$ and $M$ are two postive integers.)}\\
&\le&\frac{2|z|^{p^{(v+1)M+j}}}{\mu\left(1-\frac{1}{p^{vM+j}}\right)} \cdot \sum_{s=0}^\infty\left[p^{\beta M} p^{\beta Ms} |z|^{p^{j+(v+1)M} (p^M-1)s } \right]\\
&=& \frac{2|z|^{p^{(v+1)M+j}}}{\mu\left(1-\frac{1}{p^{vM+j}}\right)} \cdot \sum_{s=0}^\infty\left[p^{\beta M} \left( p^{\beta M} |z|^{\left( p^{(v+2)M+j}-p^{(v+1)M+j}\right) }\right)^s \right]\\
&=& \frac{2}{\mu\left(1-\frac{1}{p^{vM+j}}\right)} \cdot \frac{p^{\beta M}(|z|^{p^{vM+j}})^{p^M}}{1-p^{\beta M} |z|^{p^{vM+j} (p^{2M}-p^M)}}\\
&\le& \frac{2}{\mu\left(1-\frac{1}{p^{vM+j}}\right)} \cdot \frac{p^{\beta M} \cdot 2^{-p^{M-0.5}}}{1-p^{\beta M} \cdot 2^{-(p^{2M-0.5}-p^{M-0.5}})}\\
 &\le &\frac{1}{100\mu\left(1-\frac{1}{p^{vM+j}}\right)}.
\end{eqnarray*}

Combining all the estimates for $I_1, I_2$ and $I_3$, we get
\begin{eqnarray*}
|g_{i, j}(z)|%
&\ge& I_1-I_2-I_3 \ge \frac{1}{\mu\left(1-\frac{1}{p^{vM+j}}\right)} \left (\frac{2}{27}-\frac{1}{100}-\frac{1}{100}\right)\\
&>& \frac{1}{20\mu\left(1-\frac{1}{p^{vM+j}}\right)}=\frac{1}{20\mu\left(1-\frac{1}{p^{vM+j+\frac{1}{2}}}\right)} \cdot \frac{\mu\left(1-\frac{1}{p^{vM+j+\frac{1}{2}}}\right)}{\mu\left(1-\frac{1}{p^{vM+j}}\right)}\\
&\ge& \frac{1}{20p^{\frac{\beta}{2}}\mu\left(1-\frac{1}{p^{vM+j+\frac{1}{2}}}\right)} \ge \frac{1}{20p^{\frac{\beta}{2}}\mu(|z|)}.
\end{eqnarray*}

In summary, we have
\begin{equation} \label{abc01}
\sum_{i=1}^M \sum_{j=1}^M |g_{i, j}(z)| \ge \frac{1}{20p^{\frac{\beta}{2}}\mu(|z|)},
\end{equation}
for all $z$ such that $1-\frac{1}{p^k} \le |z| \le 1-\frac{1}{p^{k+\frac{1}{2}}}, k=1,2, \dots$.

Next, pick a sequence of positive integers $q_k$ such that $0 \le q_k-p^{k+\frac{1}{2}}<1$ and for each $1 \le j \le M$, a sequence of positive numbers $\eps_{j, v}$ such that $A^2q_{vM+j}\eps_{j, v}^2=1$.

Choose a sequence of subsets $\Psi_{j, v}$ of $\SSS$ with the following property: for each nonnegative interger $v$, the set $\bigcup\limits_{l=1}^M \Psi_{j,vM+l}$ is a maximal subset of $\SSS$ which is $d$-seperated by $A\eps_{j, v}/2$, and each $\Psi_{j, vM+l}$ is $d$-seperated by $\eps_{j, v}$.

For each $i, j=1,2, \dots, M$ and $v \ge 0$, set
$$
Q_{i, vM+j}(z)=\sum_{\xi \in \Psi_{j, vM+\tau^i(j)}} \langle z, \xi \rangle^{q_{vM+j}}
$$
and define
$$
h_{i, j}(z)=\sum_{v=0}^\infty \frac{Q_{i, vM+j}(z)}{\mu\left(1-\frac{1}{q_{vM+j}}\right)}.
$$
Then $h_{i, j}$ is in the Hadamard gap since for each $v \ge 0$,
$$
\frac{q_{vM+j}}{q_{(v-1)M+j}} \ge \frac{p^{vM+\frac{1}{2}}}{p^{(v-1)M+\frac{1}{2}}+1} \ge \frac{p^M}{2}>1.
$$
Moreover, the homogeneous polynomials $Q_{i, vM+j}$ are uniformlly bounded by $2$ as before. Hence each $h_{i, j}$ belongs to $H^\infty_\mu$ by Theorem \ref{thm002} and an easy modification of the previous arguments yields for each $v \ge 0, 1 \le j \le M$ and $z \in \BB$ satisfying
$$
1-\frac{1}{p^{vM+j+\frac{1}{2}}} \le |z| \le 1-\frac{1}{p^{vM+j+1}}
$$
there exists an index $i \in \{1, 2, \dots, M\}$, such that
$$
|h_{i, j}(z)| \ge \frac{C_p}{\mu(|z|)},
$$
where $C_p>0$.

Hence
\begin{equation} \label{abc02}
\sum_{i=1}^M \sum_{j=1}^M |h_{i, j}(z)| \ge \frac{C_p}{\mu(|z|)},
\end{equation}
for all $z$ such that $1-\frac{1}{p^{k+\frac{1}{2}}} \le |z| \le 1-\frac{1}{p^{k+1}}, k=1,2, \dots$.

Consequently, we finally have
$$
\sum_{i=1}^M \sum_{j=1}^M \left( |g_{i, j}(z)|+|h_{i, j}(z)|\right) \ge \frac{C}{\mu(|z|)}
$$
for all $z \in \BB$ sufficiently close to the boundary and for some constant $C$.  Therefore the proof is complete.  \hfill{$\Box$}\\

As a corollary, we get the following description of the growth rate on Bers-type space
$H^\infty_\alpha(\alpha>0)$,  by taking $\mu(|z|)=(1-|z|^2)^\alpha$ in Theorem \ref{thm003}.

\begin{cor}
There exists some positive integer $M$ and a sequence of functions $f_i \in H^\infty_\alpha, 1 \le i \le M$,  such that
$$
\sum_{i=1}^M |f_i(z)| \gtrsim \frac{1}{(1-|z|^2)^\alpha}, \quad z \in \BB.
$$
\end{cor}
\bigskip

\section{Weighted composition operator $uC_\varphi: H^\infty_\mu  \to H(p,q,\phi)$}

In this section, we will use Theorem 3.1 to characterize  the boundedness and compactness of the operator $uC_\varphi: H^\infty_\mu  \to H(p,q,\phi)$. Our main result is the following.\msk

\noindent{\bf Theorem 4.1.} {\it Let $\varphi$ be a holomorphic
self-map of $\mathbb{B}$ and $u\in H(\mathbb{B})$. Suppose that
$0<p,q<\infty$ and $\mu, \phi$ are normal on $[0,1)$. Then the
following statements are equivalent:

\noindent(i) The operator $uC_{\varphi}: H^\infty_\mu\to H(p,q,\phi)$
is  bounded;

\noindent(ii) The operator $uC_{\varphi}: H^\infty_\mu\to H(p,q,\phi)$
is   compact;

\noindent(iii) \begr  \int_0^1 \Big( \int_{\SSS} \frac{|u(r\xi)|^q}{\mu^q(|\varphi(r\xi)| )  }d\sigma(\xi) \Big)^{p/q}
\frac{\phi^p(r)}{1-r} dr<\infty ;\nonumber \endr

 \noindent(iv) \begr   \lim_{t \rightarrow 1}\int_0^1 \Big( \int_{|\varphi(r\xi)|>t} \frac{|u(r\xi)|^q
}{\mu^q(|\varphi(r\xi)| )  }d\sigma(\xi) \Big)^{p/q}\frac{\phi^p(r)}{1-r} dr=0. \nonumber \endr   }

Before proving Theorem 4.1, we  need the following auxiliary result, which can be proved  by standard arguments (see, e.g.,
Proposition 3.11 of \cite{cm}).  \msk

\noindent{\bf Lemma 4.1.} {\it Let $\varphi$ be a holomorphic self-map of $\mathbb{B}$ and $u\in H(\mathbb{B})$. Suppose that
$0<p,q<\infty$ and $\mu, \phi$ are normal on $[0,1)$.  Then $ uC_\varphi:H^\infty_\mu \to H(p,q,\phi)$ is compact if and only if
$uC_\varphi: H^\infty_\mu  \to H(p,q,\phi)$ is bounded and for any bounded sequence $\{f_k\}_{k\in {\NN}}$ in $H^\infty_\mu $ which
converges to zero uniformly on compact subsets of $\mathbb{B}$ as $k\to\infty$, we have $\|uC_\varphi f_k\|_{H(p,q,\phi)}\to 0 $ as
$k\to\infty.$}\\

\noindent{\bf Lemma 4.2.} \cite{shi} {\it If $a>0, b>0$, then the following elementary inequality holds.
\begr
(a+b)^p\leq  \left\{\begin{array}{ccc} a^p+b^p& , & p\in (0,1)\\
2^{p}(a^p+b^p)& , & p\geq 1
\end{array}\right. .\nonumber
\endr}

It is obvious that Lemma 4.2 holds for the sum of finite terms, that is,
$$
(a_1+\cdots +a_j)^p \leq C (a_1^p+\cdots+a_j^p),
$$
where $a_1, \cdots, a_j$ are nonnegative numbers, and $C$ is a positive constant.\\

{\bf Proof of Theorem 4.1.} $\it (ii)\Rightarrow(i)$. It is obvious.

$\it (i)\Rightarrow(iii)$. Suppose that $u C_\varphi: H^\infty_\mu  \to H(p,q,\phi)$ is bounded.  From Theorem 3.1, we
pick functions $ f_1, \cdot\cdot\cdot, f_M \in H^\infty_\mu$ such that
 \begr
\sum^M_{j=1}|f_j(z)| \gtrsim \frac{1}{\mu(|z| )}, ~~z\in\mathbb{B}.  \label{t53}
 \endr
The assumption implies that
 \begr \int_0^1 \Big( \int_{\SSS}   |(uC_\varphi f_j)(r\xi)|^q
d\sigma(\xi) \Big)^{p/q}
\frac{\phi^p(r)}{1-r} dr   <\infty, j= 1, \cdot\cdot\cdot, M, \nonumber
 \endr
which together with (\ref{t53}) and Lemma 4.2 imply
  \begr &&  \int_0^1 \Big( \int_{\SSS} \frac{|u(r\xi)|^q
}{ \mu^q(|\varphi(r\xi)| ) }d\sigma(\xi) \Big)^{p/q}
\frac{\phi^p(r)}{1-r} dr   \nonumber\\
&\lesssim  &   \int_0^1 \Big( \int_{\SSS} |u(r\xi)|^q  \Big(\sum^M_{j=1}
|f_j(\varphi(r\xi))|    \Big)^q
d\sigma(\xi) \Big)^{p/q} \frac{\phi^p(r)}{1-r} dr       \nonumber\\
&\lesssim &  \sum^M_{j=1} \int_0^1 \Big( \int_{\SSS}  |u(r\xi)|^q
|(f_j\circ\varphi)(r\xi)|^q d\sigma(\xi) \Big)^{p/q}
\frac{\phi^p(r)}{1-r} dr   \nonumber\\
&=&  \sum^M_{j=1} \int_0^1 \Big( \int_{\SSS}   |(uC_\varphi f_j)(r\xi)|^q
d\sigma(\xi) \Big)^{p/q}
\frac{\phi^p(r)}{1-r} dr    \nonumber\\
&  <&\infty, \nonumber
\endr
as desired.

$\it (iii)\Rightarrow(iv)$. This implication follows from the
dominated convergence theorem.

$\it (iv)\Rightarrow(ii)$. Assume that $\it (iv)$ holds. To prove that  $u
C_\varphi: H^\infty_\mu \to H(p,q,\phi)$ is compact, it
suffices to prove that if $\{f_k\}_{k\in {\NN}}$ is a bounded sequence in
$H^\infty_\mu $ such that $\{f_k\}_{k\in {\NN}}$ converges to zero uniformly on
compact subsets of $\mathbb{B}$, then  $ \|uC_\varphi f_k\|_{H(p,q,\phi)}\rightarrow 0, ~~\mbox{as} ~~k \rightarrow
\infty. \nonumber$  Take such a sequence $\{f_k\} \subset H^\infty_\mu $. We
have
  \begr
 &&  \int_0^1 \Big( \int_{|\varphi(r\xi)|>t}  |u(r\xi)|^q
|(f_k\circ\varphi)(r\xi)|^q d\sigma(\xi) \Big)^{p/q}
\frac{\phi^p(r)}{1-r} dr \nonumber  \\
  &\leq &  \| f_k\|^p \int_0^1 \Big( \int_{|\varphi(r\xi)|>t} \frac{|u(r\xi)|^q
}{\mu^q(|\varphi(r\xi)| )}d\sigma(\xi) \Big)^{p/q}
\frac{\phi^p(r)}{1-r} dr    \nonumber\\
   &\lesssim &    \int_0^1 \Big( \int_{|\varphi(r\xi)|>t} \frac{|u(r\xi)|^q
}{\mu^q(|\varphi(r\xi)| )}d\sigma(\xi) \Big)^{p/q}
\frac{\phi^p(r)}{1-r} dr     ,\label{t56}
\endr
for all $k$. Take $\varepsilon>0$. $\it (iv)$ and (\ref{t56}) imply
that there exists $t_0 \in (0,1)$ such that
  \begr
   &&\int_0^1 \Big( \int_{|\varphi(r\xi)|>t_0}  |u(r\xi)|^q
|(f_k\circ\varphi)(r\xi)|^q d\sigma(\xi) \Big)^{p/q}
\frac{\phi^p(r)}{1-r} dr < \varepsilon ,\label{t57}
\endr
 for all $  k$.  For the above $\varepsilon$, since $\{f_k \}$ converges to
$0$ on any compact subset of $\mathbb{B}$, there exists a $k_0$
such that
  \begr
   &&\int_0^1 \Big( \int_{|\varphi(r\xi)|\leq t_0}  |u(r\xi)|^q
|(f_k\circ\varphi)(r\xi)|^q d\sigma(\xi) \Big)^{p/q}
\frac{\phi^p(r)}{1-r} dr  < \varepsilon ,\label{t58}
\endr
  for all $k>k_0$. Hence by (\ref{t57}) and (\ref{t58})  we have
  \begr  && \|u C_\varphi f_k  \|_{H(p,q,\phi)} \nonumber\\
&=&\int_0^1 \Big( \int_{\SSS}  |u(r\xi)|^q |(f_k\circ\varphi)(r\xi)|^q
d\sigma(\xi) \Big)^{p/q}
\frac{\phi^p(r)}{1-r} dr \nonumber\\
& \lesssim&  \int_0^1 \Big( \int_{|\varphi(r\xi)|> t_0}  |u(r\xi)|^q
|(f_k\circ\varphi)(r\xi)|^q d\sigma(\xi) \Big)^{p/q}
\frac{\phi^p(r)}{1-r}
dr\nonumber\\
&  & +\int_0^1 \Big( \int_{|\varphi(r\xi)|\leq t_0}  |u(r\xi)|^q
|(f_k\circ\varphi)(r\xi)|^q d\sigma(\xi) \Big)^{p/q}
\frac{\phi^p(r)}{1-r} dr \nonumber\\
&\lesssim &   \varepsilon, ~~\mbox{as}~~k>k_0, \nonumber
\endr
  from which we obtain $\lim_{k \to\infty}\|u C_\varphi f_k
\|_{H(p,q,\phi)}=0.$ Thus $u C_\varphi: H^\infty_\mu  \to
H(p,q,\phi)$ is compact by Lemma 4.1.  This completes the proof of this theorem.\msk

{\bf Acknowledgement.}   This project  was partially supported by the
Macao Science and Technology Development Fund(No.098/2013/A3), NSF
of Guangdong Province(No.S2013010011978) and NNSF of China(No.
11471143).


\begin{thebibliography}{99}

\bibitem{ABA} A. Aleksandrov, Proper holomorphic mappings from the ball into a polydisk, \textit{Soviet
Math. Dokl.} \textbf{33} (1986), 1--5.


\bibitem{JSC} J. Choa, Some properties of analytic functions on the unit ball with Hadamard gaps,
 \textit{Complex Variables} {\bf29} (1996), 277--285.

\bibitem{CR} B. Choe and K. Rim, Fractional derivatives of Bloch functions, growth rate and interpolation,
\textit{Acta Math. Hungar.} \textbf{72(1--2)} (1996), 67--86.

\bibitem{cm} C.~Cowen and B.~MacCluer, {\it Composition Operators on Spaces of
Analytic Functions}, Studies in Advanced Math., CRC Press, Boca
Raton, 1995.

\bibitem{LS} S. Li and S Stevi\'c, Weighted-Hardy functions with Hadamard gaps on the unit ball,
 \textit{Appl. Math. Comput.} {\bf 212} (2009), 229--233.

\bibitem{JM} J. Miao, A property of analytic functions with Hadamard gaps, \textit{Bull. Austral. Math. Soc.}
{\bf45} (1992) 105--112.

\bibitem{Rud} W. Rudin, {\it Function Theory in the Unit Ball of $\CC^n$}, Springer-Verlag, New York, 1980.

\bibitem{shi} J. Shi, Inequalities for the integral means of holomorphic functions and their derivatives in the unit ball of $\CC^n$, {\it Trans. Amer. Math. Soc.} {\bf 328} (1991), 619-637.

\bibitem{SW} A. Shields and D.  Williams, Bounded projections, duality, and multipliers in spaces of analytic functions,
\textit{Trans. Amer. Math. Soc.} \textbf{162} (1971), 287--302.

\bibitem{SS} S. Stevi\'c, A generalization of a result of choa on analytic functions with Hadamard gaps,
\textit{J. Korean Math. Soc.} {\bf 43} (2006), 579--591.

\bibitem{SS1} S. Stevi\'c, Norm of weighted composition operators from Bloch space to $H_\mu^\infty$ on the
unit ball, \textit{Ars Combin.}  \textbf{88} (2008), 125--127.

\bibitem{DU} D. Ullrich, A Bloch function in the ball with no radial limits,
\textit{Bull. London Math. Soc.} \textbf{20} (1988), 337--341.

\bibitem{WZ} H. Wulan and K. Zhu, Lacunary series in $\mathcal Q_K$ spaces, \textit{Studia Math.} {\bf 178} (2007), 217--230.

\bibitem{SY} S. Yamashita, Gap series and $\alpha$-Bloch functions, \textit{Yokohama Math. J.} {\bf 28} (1980), 31--36.

\bibitem{YX} C. Yang and W. Xu, Spaces with normal weights and Hadamard gap series, \textit{Arch, Math}, \textbf{96} (2011), 151--160.

\bibitem{zz} R. Zhao and K. Zhu, Theory of Bergman spaces in the unit ball of $\mathbb{C}^n$,
{\it Memoires de la SMF}  {\bf 115} (2008), 103 pages.

\bibitem{Zhu2} K. Zhu, A class of M\"obius invariant function spaces,
{\it Illinois J. Math.}  {\bf 51} (2007), 977--1002.


\end{thebibliography}
\end{document}